\documentclass[12pt]{article}
\title{Modules determined by their composition factors in higher homological algebra}
\author{Joseph Reid}
\date{}
\usepackage{fancyhdr}
\usepackage{eqparbox}
\usepackage{amsfonts}
\usepackage{amsthm}
\usepackage{amsmath}
\usepackage{amssymb}
\usepackage{wasysym}
\usepackage{mathrsfs}
\usepackage[margin=1in]{geometry}
\theoremstyle{definition}

\newtheorem{theorem}{Theorem}[section]
\newtheorem*{theorem*}{Theorem}
\newtheorem{corollary}[theorem]{Corollary}
\newtheorem{proposition}[theorem]{Proposition}
\newtheorem{lemma}[theorem]{Lemma}
\newtheorem{remark}[theorem]{Remark}

\newtheorem{definition}[theorem]{Definition}

\usepackage{tikz}
\usetikzlibrary{arrows,positioning}
\usetikzlibrary{decorations.markings}

\tikzset{
  big dot/.style={
    circle, inner sep=0pt, 
    minimum size=3mm, fill=black
 }
}
\tikzset{
  normal dot/.style={
    circle, inner sep=0pt, 
    minimum size=1.5mm, fill=black
 }
}

\newcommand{\homs}{\textrm{Hom}_{\mathscr{C}}}
\newcommand{\C}{\mathscr{C}}
\newcommand{\T}{\mathscr{T}}

\newcommand{\homc}{\textrm{Hom}_{\mathscr{C}}}
\newcommand{\homp}{\textrm{Hom}_\Phi}
\newcommand{\IndT}{\textrm{Ind}_{\mathscr{T}}}

\begin{document}
\maketitle
\thispagestyle{fancy}
ABSTRACT. Let $\Phi$ be a finite dimensional $K$-algebra and let $\C{} = \textrm{mod}\: \Phi$ be the abelian category of finitely generated right $\Phi$-modules. In their 1985 paper ``Modules determined by their composition factors'' \cite{AusReit}, Auslander and Reiten showed that under certain conditions modules in $\textrm{mod}\: \Phi$ are determined by their composition factors, and show an important formula related to the Auslander-Reiten translation.\\

Let $\T{}$ be a $d$-cluster tilting subcategory of $\C{}$, which by definition is also $d$-abelian. In this paper we will define the Grothendieck group for a $d$-abelian category, and show that the Grothendieck groups of $\C{}$ and $\T{}$ are isomorphic. We show also that under certain conditions, the indecomposable objects of $\T{}$ are determined up to isomorphism by their composition factors in $\C{}$. Finally, we generalise the formula from Auslander and Reiten involving the higher dimensional Auslander-Reiten translation.
\section{Introduction}

The index, an invariant on categories, has been defined for triangulated categories with cluster tilting subcategories, and generalised to $(d+2)$-angulated categories with higher analogues of cluster tilting subcategories. In this paper, we will define and explore the index with respect to a $d$-cluster tilting subcategory of an abelian category. As an application, we will provide criteria under which indecomposable modules are determined by their composition factors in higher homological algebra, see Theorem \ref{tertiaryResultThm}.\\

We set up some standing notation that will be useful throughout this document. We let $K$ be an algebraically closed field, let $\Phi$ be a finite dimensional $K$-algebra and $\C{}=\textrm{mod }\Phi$ be the abelian category of finitely generated right $\Phi$-modules. We also let $d$ be a positive integer and let $\T{} = \textrm{add}(T)$ be a $d$-cluster tilting subcategory of $\C{}$, see Definition \ref{dClusterDefn}.\\

Grothendieck groups feature heavily in this paper, see Definitions \ref{splitGrothDef} and \ref{grothDef}, and we also set up some notation regarding these. For a category $\mathscr{A}$, we denote an element of the Grothendieck group of $\mathscr{A}$ as $[a]_\mathscr{A}$, where $a \in \mathscr{A}$. The notation $[-]_{\textrm{sp}}$ shows that the element belongs to a split Grothendieck group.\\

\begin{definition}\cite[Definition~2.2]{IyamaAus}\cite[Definition~3.4]{Jasso}\label{dClusterDefn}
A \textbf{$d$-cluster tilting subcategory} $\T{}$ of $\C{}$ is a full subcategory which satisfies the following.
\begin{itemize}
\item[(i)] $\begin{aligned}[t]
\T{} &= \{a \in \C{} \: | \: \textrm{Ext}_\C{}^{1..d-1}(\T{}, a)=0\} \\
&= \{a \in \C{} \: | \: \textrm{Ext}_\C{}^{1..d-1}(a,\T{})=0\}.\end{aligned}$
\item[(ii)] $\T{}$ is functorially finite.
\end{itemize}
A \textbf{$d$-cluster tilting object} of $\C{}$ is an object $T$ such that $\T{}= \textrm{add}(T)$ is a $d$-cluster tilting category.
\end{definition}

We note that $\T{}$ as defined above is a $d$-abelian category in the sense of \cite[Definition~3.1]{Jasso}. We then recall the definition of the split Grothendieck group and the Grothendieck group.
\begin{definition}\label{splitGrothDef}
Let $\mathscr{A}$ be an essentially small additive category and $G(\mathscr{A})$ be the free abelian group on isomorphism classes $[A]$ of objects $A \in \mathscr{A}$. We define the \textbf{split Grothendieck group of $\mathscr{A}$} to be
\begin{align*}
K_0^{\textrm{sp}}(\mathscr{A}):=G(\mathscr{A})/\langle [A \oplus B] - [A]- [B] \: | \: A, B \in \mathscr{A} \rangle.
\end{align*}
When $\mathscr{A}$ is abelian or triangulated, we can also define the \textbf{Grothendieck group of $\mathscr{A}$} respectively as
\begin{align*}
K_0(\mathscr{A}) &:= K_0^{\textrm{sp}}(\mathscr{A})/\langle	[A]_\textrm{sp}-[B]_\textrm{sp}+[C]_\textrm{sp}\: |\: 0 \to A \to B \to C \to 0 \textrm{ is a short exact sequence in }\mathscr{A}\rangle \\
K_0(\mathscr{A}) &:= K_0^{\textrm{sp}}(\mathscr{A})/\langle	[A]_\textrm{sp}-[B]_\textrm{sp}+[C]_\textrm{sp}\: |\: A \to B \to C \to \Sigma A \textrm{ is a triangle in }\mathscr{A} \rangle.
\end{align*}
\end{definition}
We here make a new definition, to extend the concept of Grothendieck groups to $d$-abelian categories, as defined in \cite[Definition~3.1]{Jasso}.

\begin{definition}\label{grothDef}
Let $\mathscr{A}$ be a $d$-abelian category. Then we define the \textbf{Grothendieck group of $\mathscr{A}$} as
\begin{align*}
K_0(\mathscr{A}) &:= K_0^{\textrm{sp}}(\mathscr{A})/\langle \sum_{i=0}^{d+1} (-1)^i[A_i]_\textrm{sp} \:|\: 0 \to A_{d+1} \to \ldots \to A_0 \to 0 \textrm{ is a } d\textrm{-exact sequence in } \mathscr{A} \rangle.
\end{align*}
\end{definition}

Finally, we may define the index with respect to a $d$-cluster tilting subcategory.

\begin{definition}\label{indexDef}
Recall that for each $c \in \C{}$ there is an augmented left resolution
\begin{align*}
\ldots \to 0 \to t_{d-1} \to t_{d-2} \to \ldots \to t_0 \to c \to 0
\end{align*}
with $t_i \in \T{}$, which becomes exact under the functor $\textrm{Hom}(t', -)$ for all $t' \in \T{}$; see \cite[Theorem~2.2.3]{Iyama}. \\

We define the \textbf{index} of $c$ with respect to $\T{}$ to be the map
\begin{align*}
\IndT{}: &\C{} \to K_0^{\textrm{sp}}(\T{}) \\
&c \mapsto \sum_{i=0}^{d-1}(-1)^i[t_i]_{\textrm{sp}}.
\end{align*}
\end{definition}

We will use the index to prove the following key results:
\begingroup
\renewcommand{\thetheorem}{\Alph{theorem}}
\setcounter{theorem}{0}
\begin{theorem}\label{secondaryResultThm}(Theorem \ref{isomThm})
Let $\C{}$ and $\T{}$ be as described. Then $K_0(\C{}) \cong K_0(\T{})$.
\end{theorem}

Note that the assumptions of the following result are satisfied when $\Phi$ is $d$-representation finite, see Corollary \ref{uniqueDefCor}.
\begin{theorem}\label{tertiaryResultThm}(Theorem \ref{indexPreTheorem} and Corollary \ref{finalCor})
Suppose that $d$ is odd, and $\T{}$ is such that for any two indecomposable objects $s, t \in \T{}$ at most one of the spaces $\homp{}(s, t)$ and $\homp{}(t, \tau_d s)$ is non-zero, where $\tau_d$ is the $d$-Auslander-Reiten translation, see \cite[1.4.1]{Iyama}. Then:
\begin{itemize}
\item[(i)] An indecomposable object $t \in \T{}$ is determined up to isomorphism by its class $[t]_\T{} \in K_0(\T{})$.
\item[(ii)] An indecomposable object $x \in \T{}$ is determined up to isomorphism by its composition factors in $\C{}$.
\end{itemize}

There are key technical results that are used in the proofs of Theorems \ref{secondaryResultThm} and \ref{tertiaryResultThm}, most significantly that the index is additive with an error term on short exact sequences; see Theorem \ref{primaryResultThm}. The notations $\delta^*, [-]_\Lambda$, and $\kappa$ will be defined in section 2, see Proposition \ref{isomRmk}.
\end{theorem}

\begin{theorem}\label{primaryResultThm}(Theorem \ref{errorLemma})
Let $\delta:0 \to a \to b \to c \to 0$ be an exact sequence in $\C{}$. Then
\begin{align*}
\IndT{}(a)-\IndT{}(b) + \IndT{}(c) = \kappa^{-1}([\delta^*(T)]_\Lambda).
\end{align*}
\end{theorem}

Finally, we generalise another result from Auslander and Reiten to the $d$-abelian case:
\begin{theorem}\label{otherResult}(Theorem \ref{ausReitThm})
Let two objects $s, t \in \T{}$ be given, and let $P_d \to P_{d-1} \to \ldots \to P_0$ be a truncation of the minimal projective resolution of $s$. Then we have that
\begin{align*}
\dim{} \homp{}(s,t) + (-1)^d \dim{} \homp(t, \tau_d s) = \sum_{i=0}^d \dim{} \homp{}(-1)^i(P_i, t).
\end{align*}
\end{theorem}
\endgroup

\section{$K$-theory}
Firstly, we make some remarks about the index.
\begin{remark}
The index defined in Definition \ref{indexDef} is well defined. The reason is that the index does not change when we drop trivial summands of the form
\begin{align*}
\ldots \to 0 \to t = t \to 0 \to \ldots
\end{align*}
from the augmented $\T{}$-resolution in Definition \ref{indexDef}. Doing so permits the reduction of the resolution to a minimal resolution where each differential $t_i \to t_{i-1}$ is in the radical. Hence each augmented $\T{}$-resolution gives the same index as the minimal augmented $\T{}$-resolution, which is unique up to isomorphism.
\end{remark}
We note here that for any $t \in \T{}$, the augmented left resolution is simply the sequence
\begin{align*}
\ldots \to 0 \to t = t \to 0 \to \ldots
\end{align*}
and so we have that $\IndT{}(t) = [t]_{\textrm{sp}}$.\\

We make some observations about $\T{}$. The algebra $\Lambda := \textrm{End}_\C{}(T)$ is known as a \textbf{Higher Auslander Algebra} \cite{IyamaAus}.

\begin{lemma}\label{equivLemma}
Let $\C{}$, $\T{}$ and $\Lambda$ be as above. Then there is an equivalence of categories
\begin{align*}
\T{} \xrightarrow{\textrm{Hom}_\C{}(T, -)} \textrm{proj}\:\Lambda.
\end{align*}
Specifically, the equivalence $\textrm{Hom}_\C{}(T, -)$ maps indecomposable objects in $\T{}$ to indecomposable objects in  $\textrm{proj}\:\Lambda$, and all of the indecomposables in $\textrm{proj}\:\Lambda$ are obtained in this fashion, see \cite[Proposition~2.1]{AuslanderII}.
\end{lemma}

\begin{remark}
Given the equivalence of categories detailed in Lemma \ref{equivLemma}, we have that 
\begin{align*}
K_0^{\textrm{sp}}(\T{}) &\cong K_0^{\textrm{sp}}(\textrm{proj}\:\Lambda) \\
[t]_{\textrm{sp}} &\mapsto [\homc{}(T, t)]_{\textrm{sp}}.
\end{align*}
\end{remark}

There is another important isomorphism of Grothendieck groups:
\begin{lemma}\label{hymanLemma}
Let $\C{}$ and $\T{}$ be as above, and let $\Lambda:= \textrm{End}_\C{}(T)$ be the higher Auslander algebra. The class of $M\in \textrm{mod}\: \Lambda$ in $K_0(\textrm{mod}\:\Lambda)$ is denoted $[M]_\Lambda$. Then there is an isomorphism
\begin{align*}
\rho:K_0^{\textrm{sp}}(\textrm{proj}\:\Lambda) \xrightarrow{~} K_0(\textrm{mod}\:\Lambda)
\end{align*}
where $\rho [P]_{\textrm{sp}} = [P]_\Lambda$.
\end{lemma}
\begin{proof}
We have by \cite[Section~3]{Iyama} that $\Lambda$ has finite global dimension. We may then apply \cite[Theorem~4.7]{Bass} to prove the result.
\end{proof}
Let an object $c \in \C{}$ be given. By \cite[Theorem~2.2.3]{IyamaAus} there is an augmented left $\T{}$-resolution
\begin{align*}
\ldots \to 0 \to t_{d-1} \to t_{d-2} \to \ldots \to t_0 \to c \to 0,
\end{align*}
which leads to an exact sequence
\begin{align*}
0 \to \homc{}(T, t_{d-1}) \to \ldots \to \textrm{Hom}_\C{}(T, t_0) \to \textrm{Hom}_\C{}(T, c) \to 0
\end{align*}
in $\textrm{mod}\:\Lambda$. We see that this sequence shows that
\begin{align*}
[\homc{}(T, c)]_\Lambda &= \sum_{i=0}^{d-1}(-1)^i[\homc{}(T, t_i)]_\Lambda\\
&= \rho(\sum_{i=0}^{d-1}(-1)^i[\homc{}(T, t_i)]_{\textrm{sp}}) \textrm{ where }\sum_{i=0}^{d-1}(-1)^i[\homc{}(T, t_i)]_{\textrm{sp}} \in K_0^{\textrm{sp}}(\textrm{proj}\:\Lambda) \\
&= \rho \circ \homc{}(T,\sum_{i=0}^{d-1}(-1)^i[t_i]_{\textrm{sp}}) \\
&=\rho \circ \homc{}(T,\IndT{}(c)).
\end{align*}
This shows
\begin{align}\label{kappaFormula}
[\homs{}(T, c)]_\Lambda = \rho \circ \homc{}(T,\IndT{}(c)).
\end{align} This leads to the following proposition:
\begin{proposition}\label{isomRmk}
We have that
\begin{align*}
K_0^{\textrm{sp}}(\T{}) \cong K_0(\textrm{mod}\:\Lambda)
\end{align*}
via the isomorphism $\kappa = \rho \circ \homc{}(T,-)$ from $K_0^{\textrm{sp}}(\T{})$ to $K_0(\textrm{mod}\:\Lambda)$.
\end{proposition}
\begin{proof}
We obtain this result by combining Lemma \ref{equivLemma} and Lemma \ref{hymanLemma}.
\end{proof}
We now need a tool known as the defect in order to continue our analysis of the index. This is defined as follows.
\begin{definition}\label{defectDef}
Let $\delta$: $0 \to t_{d+1} \to \ldots \to t_0 \to 0$ be a $d$-exact sequence in the $d$-abelian category $\T{}$. Then the \textbf{covariant defect} $\delta_*$ and the \textbf{contravariant defect} $\delta^*$ are defined by the exactness of the following sequences:
\begin{align*}
&0 \to \textrm{Hom}_\T{}(t_0, -) \to \ldots \to \textrm{Hom}_\T{}(t_{d+1}, -) \to \delta_*(-) \to 0 \\
&0 \to \textrm{Hom}_\Phi(-,t_{d+1}) \to \ldots \to \textrm{Hom}_\Phi(-,t_0) \to \delta^*(-) \to 0.
\end{align*}
\end{definition}

\begin{remark}
Letting $d=1$ in Definition \ref{defectDef}, we obtain the traditional defect of \cite[Section IV.4]{ARS}, where $\delta$: $0 \to A \to B \to C \to 0$ is a short exact sequence in an abelian category.
\end{remark}
We finish this section with a final result. This shows that the index is additive on short exact sequences up to an error term. This is a known result for triangulated categories, as shown by \cite[Proposition~2.2]{Palu}.
\begin{theorem}\label{errorLemma}(= Theorem \ref{primaryResultThm})
Let $\delta:0 \to a \to b \to c \to 0$ be an exact sequence in $\C{}$. Then
\begin{align*}
[\delta^*(T)]_\Lambda = \kappa(\IndT{}(a)-\IndT{}(b) + \IndT{}(c)).
\end{align*}
\end{theorem}
\begin{proof}
Let an exact sequence
\begin{align*}
\delta: 0 \to a \to b \to c \to 0
\end{align*}
be given in $\C{}$. This will give us an exact sequence
\begin{align*}
0 \to \homc{}(T, a) \to \homc{}(T, b) \to \homc{}(T, c) \to \delta^*(T) \to 0
\end{align*}
in $\textrm{mod}\:\Lambda$. This shows that 
\begin{align*}
[\delta^*(T)]_\Lambda &= [\homc{}(T, a)]_\Lambda - [\homc{}(T, b)]_\Lambda + [\homc{}(T, c)]_\Lambda \\
&=\kappa( \IndT{}(a) - \IndT{}(b) + \IndT{}(c))
\end{align*}
where the second equality is by Equation (\ref{kappaFormula}).
\end{proof}
Thus, we may determine the alternating sum of the indices of a short exact sequence in $\C{}$ by understanding the element $[\delta^*(T)]_\Lambda$ in $K_0(\textrm{mod}\:\Lambda)$. \\
\section{Morphisms}
Recalling that $\T{}$ is a $d$-cluster tilting subcategory of $\C{}$, we have that $\T{}$ is also a $d$-abelian category by \cite[Theorem~3.16]{Jasso}, and as such it has a Grothendieck group. We would like to understand the relationship between the Grothendieck groups of $\C{}$ and $\T{}$. \\

Note first that there are the canonical projections
\begin{align*}
\pi_{\C{}}: K_0^{\textrm{sp}}(\C{}) &\to K_0(\C{}) \\
\pi_{\T{}}: K_0^{\textrm{sp}}(\T{}) &\to K_0(\T{}).
\end{align*}
There is the inclusion map
\begin{align*}
\iota:K_0^{\textrm{sp}}(\T{}) &\to K_0^{\textrm{sp}}(\C{}) \\
 [t]_{\textrm{sp}} &\mapsto [t]_{\textrm{sp}}. 
\end{align*}
We may also extend Definition \ref{indexDef} to define the map
\begin{align*}
\IndT{}:K_0^{\textrm{sp}}(\C{}) &\to K_0^{\textrm{sp}}(\T{}) \\
[c]_{\textrm{sp}} &\mapsto \IndT{}(c).
\end{align*}
This gives the following diagram:
\begin{center}
\begin{tikzpicture}[line cap = round, line join = round]
\node (a) at (-2, 0) {$K_0^\textrm{sp}(\mathscr{C})$};
\node (b) at (2, 0) {$K_0^\textrm{sp}(\mathscr{T})$};

\node (c) at (-2, 2) {Ker$(\pi_{\mathscr{C}})$};
\node (d) at (2, 2) {Ker$(\pi_{\mathscr{T}})$};

\node (e) at (-2, -2) {$K_0(\mathscr{C})$};
\node (f) at (2, -2) {$K_0(\mathscr{T})$.};

\draw [->] ([yshift= 0.2 cm]a.east) to node[above]{$\IndT$} ([yshift= 0.2 cm]b.west);
\draw [->] ([yshift= -0.2 cm]b.west) to node[below]{$\iota$} ([yshift= -0.2 cm]a.east);

\draw [->] (c) to node[right]{$\kappa_{\mathscr{C}}$} (a);
\draw [->] (d) to node[right]{$\kappa_{\mathscr{T}}$} (b);
\draw [->] (a) to node[right]{$\pi_{\mathscr{C}}$} (e);
\draw [->] (b) to node[right]{$\pi_{\mathscr{T}}$} (f);
\end{tikzpicture}
\end{center}
This raises the question whether $\IndT{}$ and $\iota$ induce morphisms between $K_0(\C{})$ and $K_0(\T{})$. We will examine both. If such maps do exist, we will refer to them as  $f$ and $g$ respectively. \\

%\subsection{The Module Category Case}
%We begin by examining the case where $\C{} = \textrm{mod}\Gamma$ for some ring $\Gamma$.
%iota
We will show in Lemma \ref{inductionLemma} that $\iota$ induces a map $g: K_0(\T{}) \to K_0(\C{})$.

\begin{lemma}\cite[Proposition~4.1]{Bass}\label{longExLem}
Let the sequence
\begin{align*}
0 \to c_1 \to c_2 \to \ldots \to c_n \to 0
\end{align*}
be exact in $\C{}$. Then in $K_0(\C{})$, we have that 
\begin{align*}
\sum_{i=1}^n (-1)^i[c_i]_\C{} = 0.
\end{align*}
\end{lemma}
\begin{proof}
Omitted for brevity.
\end{proof}

\begin{lemma}\label{exactLemma}
An $d$-exact sequence in $\T{}$ is exact in $\C{}$.
\end{lemma}
\begin{proof}
Let a $d$-exact sequence
\begin{align*}
\epsilon: t_{0} \xrightarrow{\alpha} t_1 \to \ldots \to t_{d+1}
\end{align*}
in $\T{}$ be given. Then we have a right exact sequence
\begin{align*}
t_{0} \xrightarrow{\alpha} t_1 \to k \to 0
\end{align*}
where $k$ is the cokernel of $\alpha$ in $\C{}$. By \cite[3.17]{Jasso}, we may construct an exact sequence
\begin{align*}
0 \to k \to t'_2 \to \ldots \to t'_{d+1}
\end{align*}
where each $t_i' \in \T{}$. By \cite[Proof of Theorem 3.16]{Jasso} this gives a $d$-exact sequence
\begin{align*}
t_{0} \xrightarrow{\alpha} t_1 \to t'_2 \to \ldots \to t'_{d+1}
\end{align*}
which is also exact in $\C{}$. By \cite[2.8]{Jasso} we have that $d$-exact sequences are determined up to homotopy by a single morphism, and as both sequences start with $\alpha$ they have the same homotopy. Then the sequence $\epsilon$ is exact, as required.
\end{proof}

\begin{lemma}\label{inductionLemma}
The map $\iota: K_0^{\textrm{sp}}(\T{}) \to K_0^{\textrm{sp}}(\C{})$ induces a well defined map $g: K_0(\T{}) \to K_0(\C{})$.
\end{lemma}
\begin{proof}
To show this, we need to show $\pi_\C{} \circ \iota \circ \kappa_\T{} = 0$; that is, any $d$-exact sequence in $\T{}$ maps to zero via the map $\pi_\C{}$. By Lemma \ref{exactLemma}, any $d$-exact sequence is exact in $\C{}$, and by Lemma \ref{longExLem}, this means that it will map to zero via the map $\pi_\C{}$.
\end{proof}
% Index
Now we examine the index. As above, $\IndT{}$ induces a map $f:K_0(\C{}) \to K_0(\T{})$ if and only if the composition $\pi_\T{} \circ \IndT{} \circ \kappa_\C{} = 0$. We again let $\Lambda$ be the higher Auslander algebra. \\

Checking that the composition $\pi_\T{} \circ \IndT{} \circ \kappa_\C{} = 0$ amounts to checking that for an exact sequence $0 \to a \to b \to c \to 0$ in $\C{}$, the alternating sum $\IndT{}(a) - \IndT{}(b) +\IndT{}(c)$ is sent to zero by $\pi_\T{}$. Motivated by Proposition \ref{isomRmk} and Theorem \ref{errorLemma}, we start by investigating $[\delta^*(T)]_\Lambda$.\\

We know by \cite[(I, Theorem 1.7)]{ARS} that for any $m \in \textrm{mod}\:\Lambda
$ there is an equality $[m]_\Lambda = [s_1]_\Lambda + \ldots + [s_n]_\Lambda$ in $K_0(\textrm{mod}\:\Lambda)$ where $s_1, \ldots, s_n$ are the composition factors of $m$. We know that every simple module $S$ can be realised as the simple top of a projective module $P$, and we recall that all projective modules have the form $P = \homc{}(T, t)$ for some $t \in \T{}$. We wish to investigate which objects $t \in \T{}$ give rise to the composition factors of $\delta^*(T)$.
\begin{lemma}\label{compFactorLemma}
Let $\delta:0 \to a \to b \to c \to 0$ be an exact sequence in $\C{}$ with the contravariant defect $\delta^*$. Then none of the composition factors of $\delta^*(T)$ arise as the simple top of a projective module $\homc{}(T, t)$ where $t$ is projective in $\T{}$.
\end{lemma}
\begin{proof}
Let $\Lambda$ be as above, and let $\textrm{Fun}_K(\T{}, \textrm{mod}\:K)$ be the category of $K$-linear functors on $\T{}$. Then by \cite[Section~4]{AuslanderII} there is an equivalence of categories
\begin{align*}
\textrm{Fun}_K(\T{}, \textrm{mod}\:K) &\to \textrm{Mod}\:\Lambda \\
A(-) &\mapsto A(T).
\end{align*}

Then for an indecomposable $t \in \T{}$, we have an indecomposable projective $\homc{}(T, t)$ in $\textrm{mod}\:\Lambda$ which has the simple top (achieved in the manner described above) $S_t$. By the equivalence of categories described above and \cite[Proposition~2.3]{AuslanderII}, the simple top has an equivalent functor, $\sigma_t$. By the equivalence, this functor has the property that for $t' \in \T{}$, either $\sigma_t(t') = K$ when $t' \cong t$ or $\sigma_t(t')=0$ when $t$ and $t'$ are not isomorphic. Hence, $\sigma_t$ is a composition factor of $\delta^*$ if and only if $\delta^*(t) \neq 0$.\\

We have the functor $\delta^* \in \textrm{Fun}_K(\T{}, \textrm{mod}K)$. Then $\delta^*(T)$ has $S_t$ as a composition factor if and only if $\delta^*$ has $\sigma_t$ as a composition factor, and this happens if and only if $\delta^*(t) \neq 0$. By Auslander's defect formula \cite[Theorem IV.4.1]{ARS}, if $t$ is a projective module then $\delta^*(t) = 0$. Thus, $\delta^*(T)$ has no composition factors arising as the simple top of modules of the form $\homc{}(T, t)$ where $t$ is projective in $\T{}$.
\end{proof}

\begin{proposition}\label{indexLemma}
Let $\delta: 0 \to a \to b \to c \to 0$ be an exact sequence in $\C{}$. Then
\begin{align*}
\pi_\T{}(\IndT{}(a) - \IndT{}(b) +\IndT{}(c)) = 0.
\end{align*}
\end{proposition}
\begin{proof}
We have the exact sequence $\delta$ as in the statement. By Theorem \ref{errorLemma}, we have that $[\delta^*(T)]_\Lambda = \kappa(\IndT{}(a)-\IndT{}(b) + \IndT{}(c))$. We know also that $[\delta^*(T)]_\Lambda = [s_1]_\Lambda + \ldots + [s_m]_\Lambda$, where $s_1, \ldots, s_m$ are the composition factors of $\delta^*(T)$. By lemma \ref{compFactorLemma} we see that the composition factors of $\delta^*(T)$ are simple tops of modules $\homc{}(T,t)$ where $t$ is not projective in $\T{}$. Let such an object $t_0 \in \T{}$ that is indecomposable and not projective be given. By \cite[Theorem~3.3.1]{Iyama} there exists a $d$-Auslander-Reiten sequence
\begin{align*}
0 \to t_{d+1} \to t_d \to \ldots \to t_1 \xrightarrow{\alpha} t_0 \to 0
\end{align*}
in $\T{}$. In particular, this sequence is $d$-exact, and the morphism $\alpha$ is right almost split. By \cite[Definition~3.1]{Iyama}, these properties give an exact sequence
\begin{align*}
0 \to \homc{}(T, t_{d+1}) \to \homc{}(T, t_d) \to \ldots \to \homc{}(T, t_1) \to \homc{}(T, t_0) \to s \to 0
\end{align*}
where $s$ is the simple top of the projective module $\homc{}(T, t_0)$. Then we have that 
\begin{align*}
[s]_\Lambda &= \sum_{i=0}^{d+1} (-1)^i[\homc{}(T, t_i)]_\Lambda \textrm{ in }K_0(\textrm{mod}\:\Lambda) \\
&= \kappa(\sum_{i=0}^{d+1}(-1)^i[t_i]_{\textrm{sp}}).
\end{align*}
Then for each $s_j$ composition factor of $\delta^*(T)$,
\begin{align*}
[s_j]_\Lambda=\kappa(\sum_{i=0}^{d+1}(-1)^i[t_i^j]_{\textrm{sp}}))
\end{align*}
where $0 \to t_{d+1}^j \to t_d^j \to \ldots \to t_1^j \xrightarrow{\alpha} t_0^j \to 0$ is a $d$-exact sequence in $\T{}$. This shows the following:
\begin{align*}
\kappa(\IndT{}(a)-\IndT{}(b) + \IndT{}(c)) &= [\delta^*(T)]_\Lambda \\
&= [s_1]_\Lambda + \ldots +[s_m]_\Lambda \\
&= \kappa(\sum_{i=0}^{d+1}(-1)^i[t_i^1]_{\textrm{sp}})) + \ldots + \kappa(\sum_{i=0}^{d+1}(-1)^i[t_i^m]_{\textrm{sp}}))\\
\Rightarrow \IndT{}(a)-\IndT{}(b) + \IndT{}(c) &= \sum_{i=0}^{d+1}(-1)^i[t_i^1]_{\textrm{sp}} + \ldots + \sum_{i=0}^{d+1}(-1)^i[t_i^m]_{\textrm{sp}}.
\end{align*}
This proves the result, as
\begin{align*}
\pi_\T{}(\IndT{}(a)-\IndT{}(b) + \IndT{}(c)) &= \pi_\T{}(\Sigma_{i=0}^{d+1}(-1)^i[t_i^1]_{\textrm{sp}} + \ldots + \Sigma_{i=0}^{d+1}(-1)^i[t_i^m]_{\textrm{sp}}) \\
&= 0,
\end{align*}
because $\pi_\T{}$ maps the alternating sums of $d$-exact sequences to zero.
\end{proof}
\begin{lemma}
The map $\IndT{}:K_0^{\textrm{sp}}(\C{}) \to K_0^{\textrm{sp}}(\T{})$ induces a well defined map
\begin{align*}
f: &K_0(\C{}) \to K_0(\T{}),\\
&[c]_\C{} \to \IndT{}(c).
\end{align*}
\end{lemma}
\begin{proof}
To show that $f$ is well defined, it is enough to show that exact sequences in $\C{}$ are sent to zero by $\pi_\T{} \circ \IndT{}$. This is true by Proposition \ref{indexLemma}.
\end{proof}
We may update our diagram of morphisms:
\begin{center}
\begin{tikzpicture}[line cap = round, line join = round]
\node (a) at (-2, 0) {$K_0^\textrm{sp}(\mathscr{C})$};
\node (b) at (2, 0) {$K_0^\textrm{sp}(\mathscr{T})$};

\node (c) at (-2, 2) {Ker$(\pi_{\mathscr{C}})$};
\node (d) at (2, 2) {Ker$(\pi_{\mathscr{T}})$};

\node (e) at (-2, -2) {$K_0(\mathscr{C})$};
\node (f) at (2, -2) {$K_0(\mathscr{T})$.};

\draw [->] ([yshift= 0.2 cm]e.east) to node[above]{$f$} ([yshift= 0.2 cm]f.west);
\draw [->] ([yshift= -0.2 cm]f.west) to node[below]{$g$} ([yshift= -0.2 cm]e.east);

\draw [->] ([yshift= 0.2 cm]a.east) to node[above]{$\IndT$} ([yshift= 0.2 cm]b.west);
\draw [->] ([yshift= -0.2 cm]b.west) to node[below]{$\iota$} ([yshift= -0.2 cm]a.east);

\draw [->] (c) to node[right]{$\kappa_{\mathscr{C}}$} (a);
\draw [->] (d) to node[right]{$\kappa_{\mathscr{T}}$} (b);
\draw [->] (a) to node[right]{$\pi_{\mathscr{C}}$} (e);
\draw [->] (b) to node[right]{$\pi_{\mathscr{T}}$} (f);
\end{tikzpicture}
\end{center}
We show a final result concerning these morphisms.
\begin{theorem}\label{isomThm}(= Theorem \ref{secondaryResultThm})
The morphisms $f$ and $g$ are mutually inverse, and show the isomorphism $K_0(\C{}) \cong K_0(\T{})$.
\end{theorem}
\begin{proof}
We examine the compositions $f \circ g$ and $g \circ f$. Beginning with $f \circ g$, we see that for an element $[t]_\T{} \in K_0(\T{})$,
\begin{align*}
f \circ g([t]_\T{}) &= f([t]_\C{}) \\
&= [t]_\T{}.
\end{align*}
This shows that $f \circ g = \textrm{Id}_{K_0(\T{})}$. Looking at the composition $g \circ f$, we first let an object $c \in \C{}$ be given, with the augmented left $\T{}$-resolution $0 \to t_{d-1} \to \ldots \to t_0 \to c \to 0$. Then for the element $[c]_\C{} \in K_0(\C{})$,
\begin{align*}
g \circ f([c]_\C{}) &= g(\sum_{i=0}^{d-1}(-1)^i[t_i]_\T{}) \\
&=\sum_{i=0}^{d-1}(-1)^i[t_i]_\C{}\\
&=[c]_\C{}.
\end{align*}
Thus we have that $g \circ f = \textrm{Id}_{K_0(\C{})}$, and we have shown the result. 
\end{proof}
\section{Determining modules by their composition factors}
We would like to know the relationship between an object $t \in \T{}$ and the class $[t]_\T{} \in K_0(\T{})$. We begin by making the following definition.
\begin{definition}
We let \textbf{Condition H} be the condition that for two indecomposable objects $s$ and $t$ in $\T{}$, at most one of the spaces $\homp{}(s, t)$ and $\homp{}(t, \tau_d s)$ is non-zero.
\end{definition}
We will prove that under Condition H, an indecomposable object $t \in \T{}$ is determined by its class in $K_0(\C{})$. We formalise this in the following results:
\begin{lemma}\label{dependLemma}
Given $s, t \in \T{}$, the expression
\begin{align*}
\dim \homp{}(s, t) + (-1)^d \dim \homp{}(t, \tau_d s)
\end{align*}
depends only on $[t]_\T{} \in K_0(\T{})$.
\end{lemma}
\begin{proof}
To see this result, we may show that there exists a factorisation $h$:
\begin{center}
\begin{tikzpicture}[line cap = round, line join = round]
\node (a) at (-4, 2) {$K_0^\textrm{sp}(\T{})$};
\node (b) at (-4, 0) {$K_0(\mathscr{T})$};
\node (c) at (6, 2) {$\mathbb{Z}$.};

\draw [->] (a) to node[above]{$\dim \homp{}(s, -) + (-1)^d \dim \homp{}(-, \tau_d s)$} (c);
\draw [->] (a) to node[right]{$\pi_\T{}$} (b);
\draw [dashed,->] (b) to node[below]{$h$} (c);
\end{tikzpicture}
\end{center}
To show that $h$ exists, we must show that the map $G:=\dim \homp{}(s, -) + (-1)^d \dim \homp{}(-, \tau_d s)$ vanishes on $\textrm{ker}\:\pi_\T{}$, which is generated by alternating sums over the $d$-exact sequences in $\T{}$.\\

Let a $d$-exact sequence
\begin{align*}
\gamma: 0 \to t_{d+1} \to \ldots \to t_0 \to 0
\end{align*}
in $\T{}$ be given. Then we have exact sequences
\begin{align*}
0 \to \homp{}(s, t_{d+1}) \to \ldots \to \homp{}(s, t_0) \to \gamma^*(s) \to 0
\end{align*}
and
\begin{align*}
0 \to \homp{}(t_0,\tau_d s) \to \ldots \to \homp{}(t_{d+1}, \tau_d s) \to \gamma_*(\tau_d s) \to 0.
\end{align*}
This then shows that 
\begin{align*}
\dim \homp{}(s, \sum_{i=0}^{d+1}(-1)^i[t_i]_\T{}) = \sum_{i=0}^{d+1}(-1)^i \dim \homp{}(s, t_i) = \dim \gamma^*(s)
\end{align*}
and
\begin{align*}
\dim \homp{}(\sum_{i=0}^{d+1}(-1)^i[t_i]_\T{}, \tau_d s) = \sum_{i=0}^{d+1}(-1)^i \dim \homp{}(t_i, \tau_d s) = -(-1)^d\dim \gamma_*(\tau_d s).
\end{align*}
Thus, when we apply the map $G$ to the alternating sums of $K$-classes of the $d$-exact sequence $\gamma$, we obtain the expression
\begin{align*}
&\dim \gamma^*(s) + (-1)^d(-(-1)^d)\dim \gamma_*(\tau_d s) \\
= &\dim \gamma^*(s) -\dim \gamma_*(\tau_d s).
\end{align*}
By \cite[Theorem~3.8]{JassoKvamme} $\dim \gamma^*(s) -\dim \gamma_*(\tau_d s) = 0$. As such we have proven the result.
\end{proof}
\begin{theorem}\label{indexPreTheorem}(= Theorem \ref{tertiaryResultThm})
Suppose that $d$ is odd, and $\T{}$ satisfies Condition H. Then an indecomposable object $t \in \T{}$ is determined up to isomorphism by its class $[t]_\T{} \in K_0(\T{})$.
\end{theorem}
\begin{proof}
Let an indecomposable $t \in \T{}$ be given. For every indecomposable $s \in \T{}$, we have by Lemma \ref{dependLemma} that $\dim \homp{}(s, t) + (-1)^d \dim \homp{}(t, \tau_d s)$ is determined by $[t]_\T{}$. Combining this with Condition H and the fact that $(-1)^d=-1$, we see that $\dim \homp{}(s, t)$ is determined by $[t]_\T{}$. By \cite[Proposition~3.9]{Reid}, as this is true for all indecomposables $s \in \T{}$, we have determined $t$ up to isomorphism by $[t]_\T{}$.
\end{proof}

\begin{corollary}\label{ballerCor}
Suppose that $d$ is odd, and $\T{}$ satisfies Condition H. Then an indecomposable object $t \in \T{}$ is determined up to isomorphism by its class $[t]_\C{} \in K_0(\C{})$.
\end{corollary}
\begin{proof}
Theorem \ref{indexPreTheorem} shows that each $t$ is determined up to isomorphism by $[t]_\T{}$, and the isomorphism proven in Theorem \ref{isomThm} shows that $t$ is also determined by $[t]_\C{}$.
\end{proof}
The natural question arises as to where Condition H is achieved. We show a Corollary that highlights an important case where the assumption is met. This is a known result, proven by different methods by Mizuno in \cite[Theorem~1.2(2)]{Mizuno}.

\begin{corollary}\label{uniqueDefCor}
Let the finite dimensional $K$-algebra $\Phi$ be $d$-representation finite. Then an indecomposable object $t \in \T{}$ is determined up to isomorphism by its class $[t]_\C{} \in K_0(\C{})$.
\end{corollary}
\begin{proof}
To prove this result, we must prove that for two indecomposable objects $s $ and $t \in \T{}$, at most one of the spaces $\homp{}(s, t)$ and $\homp{}(t, \tau_d s)$ is non-zero. Then we may apply Corollary \ref{ballerCor} to finish the proof. \\

Let $s, t \in \textrm{Ind}\:\T{}$. By \cite[Proposition~1.3]{IyamaAus} we know that $s = \tau_d^a i$ and $t = \tau_d^b j$ for $a,b \geq 0$ and where $i,j$ are indecomposable injective modules. Then we have that 
\begin{align*}
\homp{}(s,t) &= \homp{}(\tau_d^a i, \tau_d^b j) \\
\homp{}(t, \tau_d s) &= \homp{}(\tau_d^b j, \tau_d^{a+1} i).
\end{align*}
We may investigate these spaces further using derived category theory. From \cite[Section 2]{HIO}, we have the derived Nakayama functor
\begin{align*}
\nu: D\textbf{R}\homp{}(-, \Phi): D^b(\C{}) \to D^b(\C{}),
\end{align*}
and the $d$-Nakayama functor
\begin{align*}
\nu_d: \Sigma^{-d} \circ \nu.
\end{align*}
By viewing modules as complexes concentrated at degree zero, we see that for a non-projective module $M \in \C{}$, \cite[Section~2.3]{HIO} gives that $\nu_d M =\tau_d M$; that is to say that $\nu_d M$ and $\tau_d M$ are isomorphic in the bounded derived category of $\C{}$. By assuming that $s$ is non-zero we see that $\tau_d^c i$ is non-projective for all $c < a$. Similarly, $\tau_d^c j$ is non-projective for all $c < b$. This shows that
\begin{align}\label{conditionSeq1}
\homp{}(s,t) &= \homp{}(\tau_d^a i, \tau_d^b j) = \textrm{Hom}_{D^b(\C{})}(\nu_d^a i, \nu_d^b j),
\end{align}
and that 
\begin{align}\label{conditionSeq2}
\homp{}(t, \tau_d s) &= \homp{}(\tau_d^b j, \tau_d^{a+1} i)= \textrm{Hom}_{D^b(\C{})}(\nu_d^b j, \nu_d^{a+1} i)
\end{align}
when $s$ is non-projective, and $\homp{}(t, \tau_d s) = 0$ otherwise. \\

We need a final property to complete the proof, which is the dual of \cite[Proposition~2.3(b)]{HIO}:
\begin{align*}
\textrm{Hom}_{D^b(\C{})}(\nu_d^k(D\Phi),\nu_d^l(D\Phi)) = 0 \textrm{ for } k < l.
\end{align*}
As all indecomposable injective modules are summands of $D \Phi$, we also have that
\begin{align*}
\textrm{Hom}_{D^b(\C{})}(\nu_d^k(i),\nu_d^l(j)) = 0 \textrm{ for } k < l
\end{align*}
and
\begin{align*}
\textrm{Hom}_{D^b(\C{})}(\nu_d^k(j),\nu_d^l(i)) = 0 \textrm{ for } k < l.
\end{align*}
This fact, combined with equations (\ref{conditionSeq1}) and (\ref{conditionSeq2}), shows that when $\Phi$ is $d$-representation finite then $\T{}$ satisfies Condition H. Applying Corollary \ref{ballerCor} finishes the proof.
\end{proof}
We may give a final result here. There is a classic result from homological algebra that under certain conditions, modules can be determined by their composition factors. This was shown by Auslander and Reiten in their paper ``Modules determined by their composition factors'', \cite[Section~1]{AusReit}. We generalise this result to the higher homological setting.
\begin{corollary}\label{finalCor}
Suppose that $\T{}$ satisfies Condition H. Let an indecomposable object $t \in \T{}$ be given. Then $t$ is determined up to isomorphism by its composition factors.
\end{corollary}
\begin{proof}
This is simply a combination of Corollary \ref{ballerCor} and \cite[(I, Theorem 1.7)]{ARS}.
\end{proof}

\section{Generalising a formula by Auslander and Reiten}
In this section, we will generalise a classical result from Auslander and Reiten, and use it to provide an alternative proof to Corollary \ref{ballerCor}. This is the classical result from Auslander and Reiten which states:
\begin{theorem}\cite[Theorem~1.4]{AusReit}
(a) Let $s$ and $t$ be arbitrary modules over an artin algebra $\Lambda$. Suppose that $P_1 \to P_0 \to s \to 0$ is a minimal projective presentation. Then we have
\begin{align*}
\dim \textrm{Hom}_\Lambda(s, t) - \dim \textrm{Hom}_\Lambda(t, \tau s) = \dim \textrm{Hom}_\Lambda(P_0, t) - \dim \textrm{Hom}_\Lambda(P_1, t)
\end{align*}
where $\tau$ is the Auslander-Reiten translation.
\end{theorem}

We will generalise this result for the $d$-abelian case. We note that the higher Auslander-Reiten translation, denoted by $\tau_d$, is defined as $\tau_d:= D \textrm{Tr}_d$; for more details see \cite{JassoKvamme}.
\begin{theorem}\label{ausReitThm}(= Theorem \ref{otherResult})
Let two objects $s, t \in \T{}$ be given, and let $P_d \to P_{d-1} \to \ldots \to P_0$ be a truncation of the minimal projective resolution of $s$. Then we have that
\begin{align*}
\dim{} \homp{}(s,t) + (-1)^d \dim{} \homp(t, \tau_d s) = \sum_{i=0}^d (-1)^i \dim{} \homp{}(P_i, t).
\end{align*}
\end{theorem}
\begin{proof}
We denote by $(-)^*$ the functor $\homp{}(-,\Phi)$. Then we have an exact sequence
\begin{align*}
P_0^* \to P_1^* \to \ldots \to P_d^* \to \textrm{Tr}_d s \to 0,
\end{align*}
where $\textrm{Tr}_d s$ is the $d$-th transpose of $s$. We then apply the functor $- \otimes_\Phi t$, to give the sequence
\begin{align*}
P_0^* \otimes_\Phi t \to P_1^* \otimes_\Phi t \to \ldots \to P_d^* \otimes_\Phi t \to \textrm{Tr}_d s \otimes_\Phi t \to 0.
\end{align*}
We may then use the standard tensor isomorphism to obtain the sequence
\begin{align*}
\homp{}(P_0, t) \to \homp{}(P_1, t) \to \ldots \to \homp{}(P_d, t) \to \textrm{Tr}_d s \otimes_\Phi t \to 0.
\end{align*}
We note here that due to $\T{}$ being a $d$-cluster tilting category, this is an exact sequence as we have that $\textrm{Ext}^i(s,t) = 0$ for all $0 < i < d$. We also note that as $\homp{}(-, t)$ is a left exact functor, we may augment this sequence to the long exact sequence
\begin{align*}
0\to \homp{}(s, t) \to \homp{}(P_0, t) \to \homp{}(P_1, t) \to \ldots \to \homp{}(P_d, t) \to \textrm{Tr}_d s \otimes_\Phi t \to 0.
\end{align*}
This gives us that 
\begin{align*}
\dim \homp{}(s, t) - \dim \homp{}(P_0, t) +  \ldots - (-1)^d \dim \homp{}(P_d, t) +(-1)^d \dim \textrm{Tr}_d s \otimes_\Phi t = 0,
\end{align*}
from which it follows that 
\begin{align*}
\dim \homp{}(s, t) + (-1)^d \dim \textrm{Tr}_d s \otimes_\Phi t = \sum_{i=0}^d (-1)^i \dim{} \homp{}(P_i, t).
\end{align*}
Finally, we see that $D(\textrm{Tr}_d s \otimes_\Phi t) \cong \homp{}(t, D \textrm{Tr}_d s) = \homp{}(t, \tau_d s)$, which proves the result.
\end{proof}

We may use Theorem \ref{ausReitThm} to provide an alternate proof of Corollary \ref{ballerCor}.
\begin{proof}
Let two indecomposable objects $s, t \in \T{}$ be given. We know from Theorem \ref{ausReitThm} that 
\begin{align*}
\dim{} \homp{}(s,t) + (-1)^d \dim{} \homp(t, \tau_d s) = \sum_{i=0}^d (-1)^i \dim{} \homp{}(P_i, t).
\end{align*}
We examine the sum $\sum_{i=0}^d (-1)^i \dim{} \homp{}(P_i, t)$. Let $P$ be a projective module in $\C{}$. Then we may draw a diagram of morphisms
\begin{center}
\begin{tikzpicture}[line cap = round, line join = round]
\node (a) at (-2, 2) {$\C{}$};
\node (b) at (-2, 0) {$K_0(\mathscr{C})$};
\node (c) at (2, 2) {$\mathbb{Z}$.};

\draw [->] (a) to node[above]{$\dim \homp{}(P, -)$} (c);
\draw [->] (a) to node[right]{$\rho$} (b);
\draw [dashed,->] (b) to node[below]{$k$} (c);
\end{tikzpicture}
\end{center}
To see that the factorisation $k$ exists, we note that the kernel of $\rho$ is generated by the alternating sums over exact sequences in $\C{}$. As $P$ is projective, the functor $\homp{}(P,-)$ is exact and thus $\dim \homp{}(P,-)$ sends such sums to zero. Hence the factorisation $k$ exists and shows that for a projective module $P$ and an indecomposable $t \in \T{}$, the value $\dim \homp{}(P, t)$ is determined by $[t]_\C{} \in K_0(\C{})$. \\

This means that we have determined $\dim \homp{}(s, t)$ by $[t]_\C{}$ for every indecomposable $s \in \T{}$. By \cite[Proposition~3.9]{Reid}, this determines $t$ up to isomorphism.
\end{proof}
\bibliography{mybib}{}

\begin{thebibliography}{10}

\bibitem{AuslanderII}
M.~{Auslander}.
\newblock {Representation theory of Artin algebras II}.
\newblock {\em Comm. Algebra}, 1:269--310, 1974.

\bibitem{AusReit}
M.~{Auslander} and I.~{Reiten}.
\newblock Modules determined by their composition factors.
\newblock {\em Illinois J. Math.}, 29:280--301, 1985.

\bibitem{ARS}
M.~{Auslander}, I.~{Reiten}, and S.~{Smal{$\o$}}.
\newblock {\em ``Representation Theory of Artin Algebras''}.
\newblock Cambridge University Press, Cambridge, 1997.

\bibitem{Bass}
H.~{Bass}.
\newblock {\em ``Algebraic K-theory''}.
\newblock W. A. Benjamin Inc., New York, 1968.

\bibitem{HIO}
M.~{Herschend}, O.~{Iyama}, and S.~{Oppermann}.
\newblock $n$-representation infinite algebras.
\newblock {\em Adv. Math.}, 252:292--342, 2014.

\bibitem{Iyama}
O.~{Iyama}.
\newblock Higher-dimensional {Auslander}-{Reiten} theory on maximal orthogonal
  subcategories.
\newblock {\em Adv. Math.}, 210:22--50, 2007.

\bibitem{IyamaAus}
O.~{Iyama}.
\newblock Cluster tilting for higher {Auslander} algebras.
\newblock {\em Adv. Math.}, 226:1--61, 2011.

\bibitem{Jasso}
G.~{Jasso}.
\newblock $n$-abelian and $n$-exact categories.
\newblock {\em Math. Z.}, 283 (3-4):703 -- 759, 2016.

\bibitem{JassoKvamme}
G.~{Jasso} and S.~{Kvamme}.
\newblock An introduction to higher {Auslander-Reiten} theory.
\newblock {\em Bull. Lond. Math. Soc.}, 51:1--24, 2019.

\bibitem{Mizuno}
Y.~{Mizuno}.
\newblock A {G}abriel-type theorem for cluster tilting.
\newblock {\em Proc. Lond. Math. Soc.}, 108:836--868, 2014.

\bibitem{Palu}
Y.~{Palu}.
\newblock Cluster characters for triangulated 2-calabi-yau categories.
\newblock {\em Ann. Inst. Fourier (Grenoble)}, 58:2221 -- 2248, 2008.

\bibitem{Reid}
J.~{Reid}.
\newblock Indecomposable objects determined by their index in higher
  homological algebra.
\newblock {\em Proc. Amer. Math. Soc.}, 148:2331--2343, 2020.

\end{thebibliography}

\bibliographystyle{plain}
School of Mathematics and Statistics, Newcastle University, Newcastle upon Tyne, NE1 7RU, United Kingdom \\
\textit{Email address}: j.reid4@ncl.ac.uk
\end{document}